\documentclass[10pt,twoside]{amsart} 
\usepackage{amsmath, amsthm, amscd, amsfonts, amssymb, graphicx, color}
\usepackage[bookmarksnumbered, colorlinks, plainpages]{hyperref}

\newtheorem{theorem}{Theorem}[section]
\newtheorem{lemma}[theorem]{Lemma}
\newtheorem{proposition}[theorem]{Proposition}
\newtheorem{corollary}[theorem]{Corollary}

\theoremstyle{definition}
\newtheorem{definition}[theorem]{Definition}
\newtheorem{example}[theorem]{Example}

\numberwithin{equation}{section}

\def\<{\langle}
\def\>{\rangle}

\long\def\alert#1{\smallskip{\hskip\parindent\vrule%
\vbox{\advance\hsize-2\parindent\hrule\smallskip\parindent.4\parindent%
\narrower\noindent#1\smallskip\hrule}\vrule\hfill}\smallskip}

\begin{document}
\title[Classical 2-absorbing submodules]{Classical 2-absorbing submodules of modules over commutative rings}
\author[Mostafanasab, Tekir and ULUCAK]{Hojjat Mostafanasab, \"{U}nsal Tekir and K\"{u}r\c{s}at Hakan Oral}

\subjclass[2010]{Primary: 13A15; secondary: 13C99; 13F05}
\keywords{Classical prime submodule, Classical 2-absorbing submodule.}

\begin{abstract}
In this article, all rings are commutative with nonzero identity.
Let $M$ be an $R$-module. A proper submodule $N$ of $M$ 
is called a {\it classical prime submodule}, if for each  
$m\in M$ and elements $a,b\in R$, $abm\in N$ implies
that $am\in N$ or $bm\in N$. We introduce the 
concept of ``classical 2-absorbing submodules'' as a generalization
of ``classical prime submodules''. We say that a proper submodule
$N$ of $M$ is a {\it classical 2-absorbing submodule} if whenever
$a,b,c\in R$ and $m\in M$ with $abcm\in N$,
then $abm\in N$ or $acm\in N$
or $bcm\in N$.
\end{abstract}

\maketitle

\section{\protect\bigskip Introduction} 

Throughout this paper, we assume that all rings are commutative with $1\neq 0
$. Let $R$ be a commutative ring and $M$ be an $R$-module. 
A proper
submodule $N$ of $M$ is said to be a \textit{prime submodule}, if for
each element $a\in R$ and $m\in M$, $am\in N$ implies that $m\in N$ or $%
a\in(N:_RM)=\{r\in R\mid rM\subseteq N\}$. A proper submodule $N$ of $M$ is called a \textit{classical prime submodule}, if for
each $m\in M$ and $a,b\in R$, $abm\in N$ implies that $am\in N$ or $%
bm\in N$. This notion of classical prime submodules has been extensively
studied by Behboodi in \cite{B1,B2} (see also, \cite{BK}, in which, the
notion of \textquotedblleft weakly prime submodules\textquotedblright\ is
investigated). For more information on weakly prime submodules, the reader is referred to \cite{A1,A,BSh}.

Badawi gave a generalization of prime ideals in \cite{B} and said such
ideals $2$-absorbing ideals. A proper ideal $I$ of $R$ is a \textit{%
2-absorbing ideal of} $R$ if whenever $a,b,c\in R$ and $abc\in I$, then $%
ab\in I$ or $ac\in I$ or $bc\in I$. He proved that $I$ is a $2$-absorbing
ideal of $R$ if and only if whenever $I_{1},I_{2},I_{3}$ are ideals of $R$
with $I_{1}I_{2}I_{3}\subseteq I$, then $I_{1}I_{2}\subseteq I$ or $%
I_{1}I_{3}\subseteq I$ or $I_{2}I_{3}\subseteq I$. Anderson and Badawi \cite%
{AB1} generalized the notion of $2$-absorbing ideals to $n$-absorbing
ideals. A proper ideal $I$ of $R$ is called an $n$-\textit{absorbing} (resp. 
\textit{a strongly} $n$-\textit{absorbing}) \textit{ideal} if whenever $%
x_{1}\cdots x_{n+1}\in I$ for $x_{1},\dots ,x_{n+1}\in R$ (resp. $%
I_{1}\cdots I_{n+1}\subseteq I$ for ideals $I_{1},\dots ,I_{n+1}$ of $R$),
then there are $n$ of the $x_{i}$'s (resp. $n$ of the $I_{i}$'s) whose
product is in $I$. The reader is referred to \cite{BTY,BYT,BY} for more concepts related to 2-absorbing ideals.
Yousefian Darani and Soheilnia in \cite{YF} extended $2$%
-absorbing ideals to $2$-absorbing submodules. 
A proper submodule $%
N$ of $M$ is called a {\it $2$-absorbing submodule of} $M$ if whenever $abm\in N$ for some 
$a,b\in R$ and $m\in M$, then $am\in N$ or $bm\in N$ or $ab\in \left(N:_RM\right) $. 
Generally, a proper submodule $N$ of $M$ is called an $n$-{\it absorbing submodule} if 
whenever $a_1\cdots a_nm\in N$ for $a_1,\dots a_n\in R$ and $m\in M$, then either $a_1\cdots a_n\in(N :_R M)$
or there are $n-1$ of $a_i$'s whose product with $m$ is in $N$, see \cite{YF2}. Several authors investigated
properties of $2$-absorbing submodules, for example \cite{MYTY,PB}. 

In this paper we introduce the definition of classical $2$-absorbing
submodules. A proper submodule $N$ of an $R$-module $M$ is called 
{\it classical 2-absorbing submodule} if whenever
$a,b,c\in R$ and $m\in M$ with $abcm\in N$,
then $abm\in N$ or $acm\in N$
or $bcm\in N$. Clearly, every classical prime submodule is a classical 2-absorbing submodule.
We show that every Noetherian $R$-module $M$ contains a finite number of minimal classical $2$-absorbing submodules [Theorem \ref{Noetherian}]. Further, we give the relationship between classical $2$%
-absorbing submodules, classical prime submodules and $2$-absorbing
submodules [Proposition \ref{abs-class}, Proposition \ref{multiplication}]. Moreover, we characterize classical 2-absorbing submodules in [Theorem \ref{main}, Theorem \ref{main2}]. In [Theorem \ref{product1}, Theorem \ref{product3}] we investigate classical 2-absorbing submodules of a finite direct product of modules.

\bigskip

\section{Characterizations of classical 2-absorbing submodules}

First of all we give a module which has no classical 
2-absorbing submodule.
\begin{example}
Let $p$ be a fixed prime integer and $\mathbb{N}_{0}=\mathbb{N}\cup \left\{ 0\right\}.$
Then $$E\left( p\right) :=\left\{ \alpha \in \mathbb{Q}/\mathbb{Z}\mid\alpha =\frac{r}{p^{n}}+\mathbb{Z}
\text{ \ for some }r\in\mathbb{Z}\text{ and }n\in \mathbb{N}_{0}\right\}$$ is a nonzero
submodule of the $\mathbb{Z}$-module $\mathbb{Q}/\mathbb{Z}.$ For each $t\in \mathbb{N}_{0},$ set $$G_{t}:=\left\{
\alpha \in \mathbb{Q}/\mathbb{Z}\mid\alpha =\frac{r}{p^{t}}+\mathbb{Z}\text{ \ for some }r\in \mathbb{Z}\right\}.$$ 
Notice that for each $t\in \mathbb{N}_{0}$, $G_{t}$ is a submodule of $E\left( p\right) $ generated by $\frac{1}{%
p^{t}}+\mathbb{Z}$ for each $t\in \mathbb{N}_{0}.$ Each proper submodule of $E\left(
p\right) $ is equal to $G_{i}$ for some $i\in \mathbb{N}_{0}\left( \text{see, \cite[Example 7.10]{Sh}}%
\right).$ However, no $G_{t}$ is a classical $2$-absorbing submodule of $%
E\left( p\right).$ Indeed, $\frac{1}{p^{t+3}}+\mathbb{Z}\in E\left(
p\right) $. Then $p^{3}\left( \frac{1}{p^{t+3}}+\mathbb{Z}\right) =\frac{1}{p^{t}}%
+\mathbb{Z}\in G_{t}$ but $p^{2}\left( \frac{1}{p^{t+3}}+\mathbb{Z}\right) =\frac{1}{%
p^{t+1}}+\mathbb{Z}\notin G_{t}$. 
\end{example}

\begin{theorem}
\label{im} Let $f:M\to M^{\prime}$ be an epimorphism of $R$-modules.

\begin{enumerate}
\item If $N^{\prime}$ is a classical 2-absorbing submodule of $M^{\prime}$,
then $f^{-1}(N^{\prime})$ is a classical 2-absorbing submodule of $M$.

\item If $N$ is a classical 2-absorbing submodule of $M$
containing $Ker(f)$, then $f(N)$ is a classical 2-absorbing submodule of $%
M^{\prime}$.
\end{enumerate}
\end{theorem}

\begin{proof}
$(1)$ Since $f$ is epimorphism, $f^{-1}(N^{\prime})$ is a proper submodule of $M$.
Let $a,b,c\in R$ and $m\in M$ such that $abcm\in f^{-1}(N^{\prime})$.
Then $abcf(m)\in N^{\prime}$. Hence $%
abf(m)\in N^{\prime}$ or $acf(m)\in N^{\prime}$ or $bcf(m)\in N^{\prime}$, 
and thus $abm\in f^{-1}(N^{\prime})$ or $acm\in f^{-1}(N^{\prime})$ or 
$bcm\in f^{-1}(N^{\prime})$. So, $f^{-1}(N^{\prime})$ is a classical $2$%
-absorbing submodule of $M$.\newline
$(2)$ Let $a,b,c\in R$ and $m^{\prime}\in M^{\prime}$ be such that $abcm^{\prime}\in f(N)$.
By assumption there exists $m\in M$ such that $m^{\prime}=f(m)$ and so $%
f(abcm)\in f(N)$. Since $Ker(f)\subseteq N$, we have $abcm\in N$. It implies
that $abm\in N$ or $acm\in N$ or $bcm\in N$. Hence $abm^{\prime}\in f(N)$ or 
$acm^{\prime}\in f(N)$ or $bcm^{\prime}\in f(N)$. Consequently $f(N)$ is a classical $
2$-absorbing submodule of $M^{\prime}$.
\end{proof}

As an immediate consequence of Theorem \ref{im} we have the following
corollary.

\begin{corollary}
\label{quo} Let $M$ be an $R$-module and $L\subseteq N$ be submodules of $M$%
. Then $N$ is a classical $2$-absorbing submodule of $M$ if and only if $N/L$ is a classical
2-absorbing submodule of $M/L$.
\end{corollary}

\begin{proposition}\label{intersection}
Let $M$ be an $R$-module and $N_1,~N_2$ be classical prime submodules of $M$.
Then $N_1\cap N_2$ is a classical 2-absorbing submodule of $M$.
\end{proposition}
\begin{proof}
Let for some $a,b,c\in R$ and $m\in M$, $abcm\in N_1\cap N_2$. Since
$N_1$ is a classical prime submodule, then we may assume that $am\in N_1$.
Likewise, assume that $bm\in N_2$. Hence $abm\in N_1\cap N_2$ which 
implies $N_1\cap N_2$ is a classical 2-absorbing submodule.
\end{proof}

\begin{proposition}\label{abs-class}
Let $N$ be a proper submodule of an $R$-module $M$. 
\begin{enumerate}
\item If $N$ is a 2-absorbing submodule of $M$, then $N$ is a classical 2-absorbing submodule of $M$.

\item $N$ is a classical prime submodule of $M$ if and only if $N$ is a 2-absorbing submodule of $M$
and $(N:_RM)$ is a prime ideal of $R$.
\end{enumerate}
\end{proposition}
\begin{proof}
(1) Assume that $N$ is a 2-absorbing submodule of $M$. Let $a,b,c\in R$ and $m\in M$
such that $abcm\in N$. Therefore either $acm\in N$ or $bcm\in N$ or $ab\in(N:M)$.
The first two cases lead us to the claim. In the third case we have that $abm\in N$.
Consequently $N$ is a classical 2-absorbing submodule.

(2) It is evident that if $N$ is classical prime, then it is 2-absorbing. 
Also, \cite[Lemma 2.1]{A1} implies that $(N:_RM)$ is a prime ideal of $R$. Assume that
$N$ is a 2-absorbing submodule of $M$ and $(N:_RM)$ is a prime ideal of $R$. 
Let $abm\in N$ for some $a,b\in R$ and $m\in M$ such that neither $am\in N$ nor $bm\in N$. 
Then $ab\in(N:_RM)$ and so either $a\in(N:_RM)$ or $b\in(N:_RM)$.This contradiction 
shows that $N$ is classical prime.
\end{proof}

The following example shows that the converse of Proposition \ref{abs-class}(1) is not true.
\begin{example}
Let $R=\mathbb{Z}$ and $M=\mathbb{Z}_p\bigoplus\mathbb{Z}_q\bigoplus\mathbb{Q}$ 
where $p,~q$ are two distinct prime integers. One can easily see that the zero submodule of $M$
is a classical 2-absorbing submodule. Notice that $pq(1,1,0)=(0,0,0)$, but $p(1,1,0)\not=(0,0,0)$,
$q(1,1,0)\not=(0,0,0)$ and $pq(1,1,1)\not=0$. So the zero submodule of $M$ is not 2-absorbing. 
Also, part (2) of Proposition \ref{abs-class} shows that the zero submodule is not a classical prime
submodule. Hence the two concepts of classical prime submodules and of classical 2-absorbing submodules
are different in general.
\end{example}

Let $M$ be an $R$-module and $N$ a submodule of $M$. For every $a\in R$, $\{m\in M\mid am\in N\}$ is denoted by $(N:_R a)$. It is easy to see that $(N:_Ma)$ is a submodule of $M$ containing
$N$.

\begin{theorem}\label{main}
Let $M$ be an $R$-module and $N$ be a proper submodule of $M$.
The following conditions are equivalent:
\begin{enumerate}
\item $N$ is classical 2-absorbing;

\item For every $a,b,c\in R$, $(N:_Mabc)=(N:_Mab)\cup(N:_Mac)\cup(N:_Mbc)$;

\item For every $a,b\in R$ and $m\in M$ with $abm\notin N$, $(N:_Rabm)=(N:_Ram)\cup(N:_Rbm)$;

\item For every $a,b\in R$ and $m\in M$ with $abm\notin N$, $(N:_Rabm)=(N:_Ram)$
or $(N:_Rabm)=(N:_Rbm)$;

\item For every $a,b\in R$ and every ideal $I$ of $R$ and $m\in M$ with $abIm\subseteq N$, 
either $abm\in N$ or $aIm\subseteq N$ or $bIm\subseteq N$;

\item For every $a\in R$ and every ideal $I$ of $R$ and $m\in M$ with $aIm\nsubseteq N$, $(N:_RaIm)=(N:_Ram)$
or $(N:_RaIm)=(N:_RIm)$;

\item For every $a\in R$ and every ideals $I,~J$ of $R$ and $m\in M$ with $aIJm\subseteq N$, either
$aIm\subseteq N$ or $aJm\subseteq N$ or $IJm\subseteq N$;

\item For every ideals $I,~J$ of $R$ and $m\in M$ with $IJm\nsubseteq N$, $(N:_RIJm)=(N:_RIm)$
or $(N:_RIJm)=(N:_RJm)$;

\item For every ideals $I,J,K$ of $R$ and $m\in M$ with $IJKm\subseteq N$, either
$IJm\subseteq N$ or $IKm\subseteq N$ or $JKm\subseteq N$;

\item For every $m\in M\backslash N$, $(N:_Rm)$ is a 2-absorbing ideal of $R$. 
\end{enumerate}
\end{theorem}
\begin{proof}
(1)$\Rightarrow$(2) Suppose that $N$ is a classical $2$-absorbing submodule of $M$. Let $m\in
\left( N:_Mabc\right)$. Then $abcm\in N$. Hence $abm\in N$ or $acm\in N$ or $%
bcm\in N$. Therefore $m\in\left( N:_Mab\right)$ or $m\in\left( N:_Mac\right)$
or $m\in\left( N:_Mbc\right)$. Consequently, $\left(N:_Mabc\right)=\left(N:_Mab\right)\cup\left( N:_Mac\right)\cup\left( N:_Mbc\right)$.\newline
(2)$\Rightarrow$(3) 
Let $abm\notin N$ for some $a,b\in R$ and $m\in M$. Assume that $x\in(N:_Rabm)$.
Then $abxm\in N$, and so $m\in(N:_Mabx)$. Since $abm\notin N$, $m\notin(N:_Mab)$.
Thus by part (1), $m\in(N:_Max)$ or $m\in(N:_Mbx)$, whence $x\in(N:_Ram)$
or $x\in(N:_Rbm)$. Therefore $(N:_Rabm)=(N:_Ram)\cup(N:_Rbm)$.\newline
(3)$\Rightarrow$(4) By the fact that if an ideal (a subgroup) is the union
of two ideals (two subgroups), then it is equal to one of them.\newline
(4)$\Rightarrow$(5) Let for some $a,b\in R$, an ideal $I$ of $R$ and $m\in M$, $abIm\subseteq N$.
Hence $I\subseteq(N:_Rabm)$. If $abm\in N$, then we are done. Assume that $abm\notin N$.
Therefore by part (4) we have that $I\subseteq(N:_Ram)$ or $I\subseteq(N:_Rbm)$, i.e., $aIm\subseteq N$
or $bIm\subseteq N$.\newline
(5)$\Rightarrow$(6)$\Rightarrow$(7)$\Rightarrow$(8)$\Rightarrow$(9) Have proofs similar to that of the previous implications.\newline
(9)$\Rightarrow$(1) Is trivial.\newline
(9)$\Leftrightarrow$(10) Straightforward.
\end{proof}

\begin{corollary}
Let $R$ be a ring and $I$ be a proper ideal of $R$.
\begin{enumerate}
\item $_{R}I$ is a classical 2-absorbing submodule of $R$ if and only if $I$ is a 2-absorbing ideal of $R$. 

\item Every proper ideal of $R$ is 2-absorbing if and only if for every $R$-module $M$
and every proper submodule $N$ of $M$, $N$ is a classical 2-absorbing submodule of $M$.
\end{enumerate} 
\end{corollary}
\begin{proof}
(1) Let $I$ be a classical 2-absorbing submodule of $R$. Then by Theorem \ref{main},
$(I:_R1)=I$ is a 2-absorbing ideal of $R$. For the converse see part (1) of Proposition \ref{abs-class}.

(2) Assume that every proper ideal of $R$ is 2-absorbing.
Let $N$ be a proper submodule of an $R$-module $M$.
Since for every $m\in M\backslash N$, $(N:_Rm)$ is a proper ideal of $R$, then it is a 2-absorbing ideal of $R$. 
Hence by Theorem \ref{main}, $N$ is a classical 2-absorbing submodule of $M$. We have the converse immediately 
by part (1).
\end{proof}

\begin{proposition}\label{inter}
Let $M$ be an $R$-module and $\left\{ K_{i}\mid i\in I\right\}$ be a chain of classical 2-absorbing submodules of $M$.
Then $\cap_{i\in I}K_{i}$ is a classical 2-absorbing submodule of $M$.
\end{proposition}
\begin{proof}
Suppose that $abcm\in \cap_{i\in I}K_{i}$ for some $a,b,c\in R$ and $m\in M$. Aassume that 
$abm\notin\cap_{i\in I}K_{i}$ $\ $and $acm\notin \cap_{i\in I}K_{i}$. Then there are $t,~l\in I$ where $abm\notin K_{t}$ and $acm\notin K_{l}$.
Hence, for every $K_{s}\subseteq K_{t}$ and every $K_{d}\subseteq K_{l}$ we have that $abm\notin K_{s}$ and $acm\notin K_{d}$. Thus, for every submodule
$K_h$ such that $K_{h}\subseteq K_{t}$ and $K_{h}\subseteq
K_{l}$ we get $bcm\in K_{h}$. Hence $bcm\in\cap_{i\in I}K_{i}$.
\end{proof}

A classical $2$-absorbing submodule of $M$ is called {\it minimal}, if for any
classical $2$-absorbing submodule $K$ of $M$ such that $K\subseteq N$, then $K=N$%
. Let $L$ be a classical $2$-absorbing submodule of $M$. Set $$\Gamma
=\left\{ K\mid K\text{ is a classical }2\text{-absorbing submodule of }M\text{ 
and }K\subseteq L\right\}.$$ If $\left\{ K_{i}:i\in I\right\} $ is any
chain in $\Gamma $, then $\cap _{i\in I}K_{i}$ \ is in $\Gamma $, by Proposition \ref{inter}.
By Zorn's Lemma, $\Gamma $ contains a minimal member which is clearly a minimal
classical $2$-absorbing submodule of $M$. Thus, every classical $2$%
-absorbing submodule of $M$ contains a minimal classical $2$-absorbing
submodule of $M$. If $M$ is a finitely generated, then it is clear that $M$
contains a minimal classical $2$-absorbing submodule.

\begin{theorem}\label{Noetherian}
\bigskip Let $M$ be a Noetherian $R$-module. Then $M$ contains a finite
number of minimal classical $2$-absorbing submodules.
\end{theorem}

\begin{proof}
Suppose that the result is false. Let $\Gamma $ denote the collection of
proper submodules $N$ of $M$ such that the module $M/N$ has an infinite
number of minimal classical $2$-absorbing submodules. Since $0\in \Gamma $
we get $\Gamma \neq \varnothing $. Therefore $\Gamma $ has a maximal member $%
T$, since $M$ is a Noetherian $R$-module. It is clear that $T$ is not a
classical $2$-absorbing submodule. Therefore, there exists an element $m\in
M\backslash T$ and ideals $I$, $J$, $K$ in $R$ such that $IJKm\subseteq T$
but $IJm\nsubseteq T$, $IKm\nsubseteq T$ and $JKm\nsubseteq T$. The
maximality of $T$ implies that $M/\left( T+IJm\right) $, $M/\left(
T+IKm\right) $ and $M/\left( T+JKm\right)$ have only finitely many minimal
classical 2-absorbing submodules. Suppose $P/T$ be a minimal classical $2$%
-absorbing submodule of $M/T$. So $IJKm\subseteq T\subseteq P$, which
implies that $IJm\subseteq P$ or $IKm\subseteq P$ or $JKm\subseteq P$.
Thus $P/\left( T+IJm\right) $ is a minimal classical 2-absorbing submodule
of $M/\left( T+IJm\right) $ or $P/\left( T+IKm\right) $ is a minimal
classical 2-absorbing submodule of $M/\left( T+IKm\right) $ or $P/\left(
T+JKm\right) $ is a minimal classical 2-absorbing submodule of $M/\left(
T+JKm\right) $. Thus, there are only a finite number of possibilities for
the submodule $P$. This is a contradiction.
\end{proof}

We recall from \cite{B} that if $I$ is a 2-absorbing ideal of a ring $R$, then either $\sqrt{I}=P$ where $P$
is a prime ideal of $R$ or $\sqrt{I}=P_1\cap P_2$ where $P_1,P_2$
are the only distinct minimal prime ideals of $I$.
\begin{corollary}
Let $N$ be a classical 2-absorbing submodule of an $R$-module $M$. Suppose that $m\in M\backslash N$
and $\sqrt{(N:_Rm)}=P$ where $P$ is a prime ideal of $R$ and $(N:_Rm)\neq P$.
Then for each $x\in\sqrt{(N:_Rm)}\backslash(N:_Rm)$, $(N:_Rxm)$ is a prime ideal of $R$ containing 
$P$. Furthermore, either $(N:_Rxm)\subseteq(N:_Rym)$ or $(N:_Rym)\subseteq(N:_Rxm)$
for every $x,y\in\sqrt{(N:_Rm)}\backslash(N:_Rm)$.
\end{corollary}
\begin{proof}
By Theorem \ref{main} and \cite[Theorem 2.5]{B}.
\end{proof}

\begin{corollary}
Let $N$ be a classical 2-absorbing submodule of an $R$-module $M$. Suppose that $m\in M\backslash N$ and $\sqrt{(N:_Rm)}=P_1\cap P_2$
where $P_1$ and $P_2$ are the only nonzero distinct prime ideals of $R$ that are minimal over $(N:_Rm)$.
Then for each $x\in\sqrt{(N:_Rm)}\backslash(N:_Rm)$, $(N:_Rxm)$ is a prime ideal of $R$ containing 
$P_1$ and $P_2$. Furthermore, either $(N:_Rxm)\subseteq(N:_Rym)$ or $(N:_Rym)\subseteq(N:_Rxm)$
for every $x,y\in\sqrt{(N:_Rm)}\backslash(N:_Rm)$.
\end{corollary}
\begin{proof}
By Theorem \ref{main} and \cite[Theorem 2.6]{B}.
\end{proof}

An $R$-module $M$ is called a 
\textit{multiplication module} if every submodule $N$ of $M$ has the form $IM
$ for some ideal $I$ of $R$. Let $N$ and $K$ be submodules of a multiplication $R$%
-module $M$ with $N=I_{1}M$ and $K=I_{2}M$ for some ideals $I_{1}$ and $I_{2}
$ of $R$. The product of $N$ and $K$ denoted by $NK$ is defined by $%
NK=I_{1}I_{2}M$. Then by \cite[Theorem 3.4]{Am}, the product of $N$ and $K$
is independent of presentations of $N$ and $K$.

\begin{proposition}
Let $M$ be a multiplication $R$-module and $N$ be a proper submodule of $M$.
The following conditions are equivalent:

\begin{enumerate}
\item $N$ is a classical 2-absorbing submodule of $M$;

\item If $N_1N_2N_3m\subseteq N$ for some submodules $N_1,N_2,N_3$ of $
M$ and $m\in M$, then either $N_1N_2m\subseteq N$ or $N_1N_3m\subseteq N$ or $N_2N_3m\subseteq N$.
\end{enumerate}
\end{proposition}

\begin{proof}
(1)$\Rightarrow$(2) Let $N_1N_2N_3m\subseteq N$ for some submodules $N_1,N_2,N_3$
of $M$ and $m\in M$. Since $M$ is multiplication, there are ideals $I_1,I_2,I_3$ of $R$ such that $N_1=I_1M$, $%
N_2=I_2M$ and $N_3=I_3M$. Therefore $I_1I_2I_3m\subseteq N$, and 
so either $I_1I_2m\subseteq N$ or $I_1I_3m\subseteq N$ or $I_2I_3m\subseteq N$.
Hence $N_1N_2m\subseteq N$ or $N_1N_3m\subseteq N$ or $N_2N_3m\subseteq N$.\newline
(2)$\Rightarrow$(1) Suppose that $I_{1}I_{2}I_3m\subseteq N$ for some ideals $%
I_{1},I_{2},I_3$ of $R$ and some $m\in M$. It is sufficient to set $%
N_1:=I_1M$, $N_2:=I_2M$ and $N_3=I_3M$ in part (2).
\end{proof}

In \cite{Q}, Quartararo et al. said that a commutative ring $R$ is a $u$%
-ring provided $R$ has the property that an ideal contained in a finite
union of ideals must be contained in one of those ideals; and a $um$-ring is
a ring $R$ with the property that an $R$-module which is equal to a finite
union of submodules must be equal to one of them. 
They show that every B$\acute{\rm e}$zout ring is a $u$-ring. Moreover, they proved that 
every Pr\"{u}fer domain is a $u$-domain. Also, any ring which contains an infinite field
as a subring is a $u$-ring, \cite[Exercise 3.63]{Sh}.

\begin{theorem}\label{main2}
Let $R$ be a $um$-ring, $M$ be an $R$-module and $N$ be a proper submodule of $M$.
The following conditions are equivalent:
\begin{enumerate}
\item $N$ is classical 2-absorbing;

\item For every $a,b,c\in R$, $(N:_Mabc)=(N:_Mab)$ or $(N:_Mabc)=(N:_Mac)$
or $(N:_Mabc)=(N:_Mbc)$;

\item For every $a,b,c\in R$ and every submodule $L$ of $M$, $abcL\subseteq N$
implies that $abL\subseteq N$ or $acL\subseteq N$ or $bcL\subseteq N$;

\item For every $a,b\in R$ and every submodule $L$ of $M$ with $abL\nsubseteq N$, $(N:_RabL)=(N:_RaL)$
or $(N:_RabL)=(N:_RbL)$;

\item For every $a,b\in R$, every ideal $I$ of $R$ and every submodule $L$ of $M$, $abIL\subseteq N$
implies that $abL\subseteq N$ or $aIL\subseteq N$ or $bIL\subseteq N$;

\item For every $a\in R$, every ideal $I$ of $R$ and every submodule $L$ of $M$ with $aIL\nsubseteq N$, $(N:_RaIL)=(N:_RaL)$
or $(N:_RaIL)=(N:_RIL)$;

\item For every $a\in R$, every ideals $I,~J$ of $R$ and every submodule $L$ of $M$, $aIJL\subseteq N$
implies that $aIL\subseteq N$ or $aJL\subseteq N$ or $IJL\subseteq N$;

\item For every ideals $I,~J$ of $R$ and every submodule $L$ of $M$ with $IJL\nsubseteq N$, $(N:_RIJL)=(N:_RIL)$
or $(N:_RIJL)=(N:_RJL)$;

\item For every ideals $I,J,K$ of $R$ and every submodule $L$ of $M$, $IJKL\subseteq N$
implies that $IJL\subseteq N$ or $IKL\subseteq N$ or $JKL\subseteq N$;

\item For every submodule $L$ of $M$ not contained in $N$, $(N:_RL)$ is a 2-absorbing ideal of $R$. 
\end{enumerate}
\end{theorem}
\begin{proof}
Similar to the proof of Theorem \ref{main}.
\end{proof}

\begin{proposition}
Let $R$ be a $um$-ring and $N$ be a proper submodule of an $R$-module $M$. 
Then $N$ is a classical 2-absorbing submodule of $M$ if and only if $N$ is a 4-absorbing submodule of $M$
and $(N:_RM)$ is a 2-absorbing ideal of $R$.
\end{proposition}
\begin{proof}
It is trivial that if $N$ is classical 2-absorbing, then it is 4-absorbing. 
Also, Theorem \ref{main2} implies that $(N:_RM)$ is a 2-absorbing ideal of $R$. Now, assume that
$N$ is a 4-absorbing submodule of $M$ and $(N:_RM)$ is a 2-absorbing ideal of $R$. 
Let $a_1a_2a_3m\in N$ for some $a_1,a_2,a_3\in R$ and $m\in M$ such that neither $a_1a_2m\in N$ nor $a_1a_3m\in N$
nor $a_2a_3m\in N$. Then $a_1a_2a_3\in(N:_RM)$ and so either $a_1a_2\in(N:_RM)$ or $a_1a_3\in(N:_RM)$
or $a_2a_3\in(N:_RM)$.This contradiction shows that $N$ is classical 2-absorbing.
\end{proof}

\begin{proposition}
Let $M$ be an $R$-module and $N$ be a classical 2-absorbing submodule of $M$.
The following conditions hold:
\begin{enumerate}
\item For every $a,b,c\in R$ and $m\in M$, $(N:_Rabcm)=(N:_Rabm)\cup(N:_Racm)\cup(N:_Rbcm)$;

\item If $R$ is a $u$-ring, then for every $a,b,c\in R$ and $m\in M$,
$(N:_Rabcm)=(N:_Rabm)$ or $(N:_Rabcm)=(N:_Racm)$ or $(N:_Rabcm)=(N:_Rbcm)$.
\end{enumerate}
\end{proposition}
\begin{proof}
(1) Let $a,b,c\in R$ and $m\in M$. Suppose that $r\in(N:_Rabcm)$. Then $abc(rm)\in N$. So,
either $ab(rm)\in N$ or $ac(rm)\in N$ or $bc(rm)\in N$. Therefore, either
$r\in(N:_Rabm)$ or $r\in(N:_Racm)$ or $r\in(N:_Rbcm)$. Consequently
$(N:_Rabcm)=(N:_Rabm)\cup(N:_Racm)\cup(N:_Rbcm)$.

(2) Use part (1).
\end{proof}

\begin{proposition}\label{multiplication}
Let $R$ be a $um$-ring, $M$ be a multiplication $R$-module and $N$ be a proper submodule of $M$.
The following conditions are equivalent:

\begin{enumerate}
\item $N$ is a classical 2-absorbing submodule of $M$;

\item If $N_1N_2N_3N_4\subseteq N$ for some submodules $N_1,N_2,N_3,N_4$ of $
M$, then either $N_1N_2N_4\subseteq N$ or $N_1N_3N_4\subseteq N$ or $N_2N_3N_4\subseteq N$;

\item If $N_1N_2N_3\subseteq N$ for some submodules $N_1,N_2,N_3$ of $
M$, then either $N_1N_2\subseteq N$ or $N_1N_3\subseteq N$ or $N_2N_3\subseteq N$;

\item $N$ is a 2-absorbing submodule of $M$;

\item $(N:_RM)$ is a 2-absorbing ideal of $R$.
\end{enumerate}
\end{proposition}

\begin{proof}
(1)$\Rightarrow$(2) Let $N_1N_2N_3N_4\subseteq N$ for some submodules $N_1,N_2,N_3,N_4$
of $M$. Since $M$ is multiplication, there are ideals $I_1,I_2,I_3$ of $R$ such that $N_1=I_1M$, $%
N_2=I_2M$ and $N_3=I_3M$. Therefore $I_1I_2I_3N_4\subseteq N$, and 
so $I_1I_2N_4\subseteq N$ or $I_1I_3N_4\subseteq N$ or $I_2I_3N_4\subseteq N$.
Thus by Theorem \ref{main2}, either $N_1N_2N_4\subseteq N$ or $N_1N_3N_4\subseteq N$ or $N_2N_3N_4\subseteq N$.\newline
(2)$\Rightarrow$(3) Is easy.\newline
(3)$\Rightarrow$(4) Suppose that $I_{1}I_{2}K\subseteq N$ for some ideals $%
I_{1},~I_{2}$ of $R$ and some submodule $K$ of $M$. It is sufficient to set $%
N_1:=I_1M$, $N_2:=I_2M$ and $N_3=K$ in part (3).\newline
(4)$\Rightarrow$(1) By part (1) of Proposition \ref{abs-class}.\newline
(4)$\Rightarrow$(5) By \cite[Theorem 2.3]{PB}. \newline
(5)$\Rightarrow$(4) Let $I_1I_2K\subseteq N$ for some ideals $I_1,~I_2$ of $R$ and some
submodule $K$ of $M$. Since $M$ is multiplication, then there is an ideal $%
I_3$ of $R$ such that $K=I_3M$. Hence $I_1I_2I_3\subseteq (N:_{R}M)$ which
implies that either $I_1I_2\subseteq (N:_{R}M)$ or $I_1I_3\subseteq {%
(N:_{R}M)}$ or $I_2I_3\subseteq {(N:_{R}M)}$.
If $I_1I_2\subseteq (N:_{R}M),$ then we are done. So, suppose
that $I_1I_3\subseteq{(N:_{R}M)}$. Thus $I_1I_3M=I_1K\subseteq N$. Similary if $I_2I_3\subseteq {%
(N:_{R}M)}$, then we have $I_2K\subseteq N$.
\end{proof}

\begin{definition} 
Let $R$ be a um-ring, $M$ be an $R$-module and $S$ be a subset of $M\backslash
\left\{ 0\right\} $. If for all ideals $I$, $J$, $Q$ of $R$ and all
submodules $K$, $L$ of $M$, $\left( K+IJL\right) \cap S\neq \emptyset $ and $%
\left( K+IQL\right) \cap S\neq \emptyset $ and $\left( K+JQL\right) \cap
S\neq \emptyset $ implies $\left( K+IJQL\right) \cap S\neq \emptyset $, then the subset $%
S$ is called {\it classical 2-absorbing m-closed}.
\end{definition}

\begin{proposition}
Let $R$ be a um-ring, $M$ be $R$-module and $N$ a submodule of $M$. Then $N$
is a classical 2-absorbing submodule if and only if $M\backslash N$ is a
classical 2-absorbing m-closed.
\end{proposition}

\begin{proof}
Suppose that $N$ is a classical 2-absorbing submodule of $M$ and $I$, $J$, $Q
$ are ideals of $R$ and $K$, $L$ are submodules of $M$ such that $\left(
K+IJL\right) \cap S\neq \emptyset $ and $\left( K+IQL\right) \cap S\neq
\emptyset $ and $\left( K+JQL\right) \cap S\neq \emptyset $ where $%
S=M\backslash N$. Assume that $\left( K+IJQL\right) \cap S=\emptyset $. Then
$K+IJQL\subseteq N$ and so $K\subseteq N$ and $IJQL\subseteq N$. Since $N$
is a classical 2-absorbing submodule, we get $IJL\subseteq N$ or $IQL\subseteq N
$ or $JQL\subseteq N$. If $IJL\subseteq N$, then we get $\left( K+IJL\right)
\cap S=\emptyset $, since $K\subseteq N$. This is a contradiction. By the
other cases we get similar contradictions. Now for the converse suppose that $%
S=M\backslash N$ is a classical 2-absorbing m-closed and assume that $%
IJQL\subseteq N$ for some ideals $I$, $J$, $Q$ of $R$ and submodule $L$ of $%
M$. Then we get for submodule $K=(0)$, $ K+IJQL \subseteq N$.
Thus $\left( K+IJQL\right) \cap S=\emptyset $. Since $S$ is a classical
2-absorbing m-closed, $\left( K+IJL\right) \cap S=\emptyset $ or $\left(
K+IQL\right) \cap S=\emptyset $ or $\left( K+JQL\right) \cap S=\emptyset $.
Hence $IJL\subseteq N$ or $IQL\subseteq N$ or $JQL\subseteq N$. So $N$ is a
classical 2-absorbing submodule.
\end{proof}

\begin{proposition}
Let $R$ be a um-ring, $M$ be an $R$-module, $N$ a submodule of $M$ and $%
S=M\backslash N$. The following conditions are equivalent:
\begin{enumerate}
\item $N$ is a classical 2-absorbing submodule of $M$;
\item $S$ is a classical 2-absorbing m-closed;
\item For every ideals $I$, $J$, $Q$ of $R$ and every submodule $L$ of $M$,
if $IJL\cap S\neq \emptyset $ and $IQL\cap S\neq \emptyset $ and $JQL\cap
S\neq \emptyset $, then $IJQL\cap S\neq \emptyset $;
\item For every ideals $I$, $J$, $Q$ of $R$ and every $m\in M$, if $IJm\cap
S\neq \emptyset $ and $IQm\cap S\neq \emptyset $ and $JQm\cap S\neq
\emptyset $, then $IJQm\cap S\neq \emptyset $.
\end{enumerate}
\end{proposition}

\begin{proof}
It follows from the previous Proposition, Theorem \ref{main} and Theorem \ref{main2}.
\end{proof}

\begin{theorem}
Let $R$ be a um-ring, $M$ be an $R$-module and $S$ be a classical 2-absorbing
m-closed. Then the set of all submodules of $M$ which are disjoint from $S$
has at least one maximal element. Any such maximal element is a classical
2-absorbing submodule.
\end{theorem}

\begin{proof}
Let $\Psi =\left\{ N\mid N\text{ is a submodule of }M\text{ and }N\cap
S=\emptyset \right\} $. Then $\left( 0\right) \in \Psi \neq \emptyset $.
Since $\Psi $ is partially ordered by using Zorn's Lemma we get at least a
maximal element of $\Psi $, say $P$, with property $P\cap S=\emptyset $. Now
we will show that $P$ is classical 2-absorbing. Suppose that $IJQL\subseteq P
$ for ideals $I$, $J$, $Q$ of $R$ and submodule $L$ of $M$. Assume that $%
IJL\nsubseteq P$ or $IQL\nsubseteq P$ or $JQL\nsubseteq P$. Then by the
maximality of $P$ we get $\left( IJL+P\right) \cap S\neq \emptyset $ and $%
\left( IQL+P\right) \cap S\neq \emptyset $ and $\left( JQL+P\right) \cap
S\neq \emptyset $. Since $S$ is a classical 2-absorbing m-closed we have $%
\left( IJQL+P\right) \cap S\neq \emptyset $. Hence $P\cap S\neq \emptyset $, which
is a contradiction. Thus $P$ is a classical 2-absorbing submodule of $M$.
\end{proof}

\begin{theorem}\label{flat}
Let $R$ be a $um$-ring and $M$ be an $R$-module.
\end{theorem}
\begin{enumerate}
\item If $F$ is a flat $R$-module and $N$ is a classical $2$%
-absorbing submodule of $M$ such that $F\otimes N\neq F\otimes M,$ then $%
F\otimes N$ is a classical $2$-absorbing submodule of $F\otimes M.$

\item Suppose that $F$ is a faithfully flat $R$-module. Then $N$ is a
classical $2$-absorbing submodule of $M$ if and only if $F\otimes N$ is a
classical $2$-absorbing submodule of $F\otimes M.$
\end{enumerate}
\begin{proof}
$\left( 1\right) $ Let $a,b,c\in R$. Then we get by Theorem \ref{main2}, $\left(
N:_Mabc\right) =\left( N:_Mab\right) $ \ or $\left( N:_Mabc\right) =\left(
N:_Mac\right) $ \ or $\left( N:_Mabc\right) =\left( N:_Mbc\right) $. Assume that $%
\left( N:_Mabc\right) =\left( N:_Mab\right) $. Then by \cite[Lemma 3.2]{A}, $\left( F\otimes
N:_{F\otimes M}abc\right) =F\otimes \left( N:_Mabc\right) =F\otimes \left( N:_Mab\right)
=\left( F\otimes N:_{F\otimes M}ab\right) $. Again Theorem \ref{main2} implies that $F\otimes N$ is a
classical 2-absorbing submodule of $F\otimes M.$

$\left( 2\right) $ Let $N$ be a classical 2-absorbing submodule of $M$
and assume that $F\otimes N=F\otimes M$. Then $0\rightarrow F\otimes N%
\overset{\subseteq }{\rightarrow }F\otimes M\rightarrow 0$ is an exact
sequence. Since $F$ is a faithfully flat module, $0\rightarrow N\overset{%
\subseteq }{\rightarrow }M\rightarrow 0$ is an exact sequence. So $N=M$,
which is a contradiction. So $F\otimes N\neq F\otimes M$. Then $F\otimes N$
is a classical 2-absorbing submodule by $\left( 1\right) $. Now for
conversely, let $F\otimes N$ be a classical 2-absorbing submodule of $%
F\otimes M$. We have $F\otimes N\neq F\otimes M$ and so $N\neq M$. Let $%
a,b,c\in R$. Then $\left( F\otimes N:_{F\otimes M}abc\right) =\left( F\otimes
N:_{F\otimes M}ab\right) $ or $\left( F\otimes N:_{F\otimes M}abc\right) =\left( F\otimes
N:_{F\otimes M}ac\right) $ or $\left( F\otimes N:_{F\otimes M}abc\right) =\left( F\otimes N:_{F\otimes M}bc\right) $
by Theorem \ref{main2}. Assume that $\left( F\otimes N:_{F\otimes M}abc\right) =\left( F\otimes
N:_{F\otimes M}ab\right) $. Hence $F\otimes \left( N:_Mab\right) =\left( F\otimes
N:_{F\otimes M}ab\right) =\left( F\otimes N:_{F\otimes M}abc\right) =F\otimes \left( N:_Mabc\right) $.
So $0\rightarrow F\otimes \left( N:_Mab\right) \overset{\subseteq }{%
\rightarrow }F\otimes \left( N:_Mabc\right) \rightarrow 0$ is an exact
sequence. Since $F$ is a faithfully flat module, $0\rightarrow \left(
N:_Mab\right) \overset{\subseteq }{\rightarrow }\left( N:_Mabc\right)
\rightarrow 0$ is an exact sequence which implies that $\left( N:_Mab\right)
=\left( N:_Mabc\right) $. Consequently $N$ is a classical 2-absorbing
submodule of $M$ by Theorem \ref{main2}.
\end{proof}

\begin{corollary}
Let $R$ be a $um$-ring, $M$ be an $R$-module and $X$ be an indeterminate. If $N$ is a classical 2-absorbing submodule of $M$,
then $N[X]$ is a classical 2-absorbing submodule of $M[X]$.
\end{corollary}
\begin{proof}
Assume that $N$ is a classical 2-absorbing submodule of $M$. Notice that $R[X]$ is a flat $R$-module.
So by Theorem \ref{flat}, $R[X]\otimes N\simeq N[X]$ is a classical 2-absorbing submodule of $R[X]\otimes M\simeq M[X]$.
\end{proof}

For an $R$-module $M$, the set of zero-divisors of $M$ is denoted by $Z_R(M)$.
\begin{proposition}
Let $M$ be an $R$-module, $N$ be a submodule and $S$ be a multiplicative
subset of $R$. 
\begin{enumerate}
\item If $N$ is a classical 2-absorbing submodule of $M$ such that $\left( N:_RM\right)\cap S=\emptyset$,
then $S^{-1}N$ is a classical 2-absorbing submodule of $S^{-1}M$.
\item If $S^{-1}N$ is a classical 2-absorbing submodule of $S^{-1}M$ such that $Z_R(M/N)\cap S=\emptyset$,
then $N$ is a classical 2-absorbing submodule of $M$.
\end{enumerate}
\end{proposition}

\begin{proof}
(1) Let $N$ be a classical 2-absorbing submodule of $M$ and $\left( N:_RM\right) \cap S=\emptyset $%
. Suppose that $\frac{a_{1}}{s_{1}}\frac{a_{2}}{s_{2}}\frac{a_{3}}{s_{3}}%
\frac{m}{s_{4}}\in S^{-1}N$. Then there exist $n\in N$ and $s\in S$ such
that $\frac{a_{1}}{s_{1}}\frac{a_{2}}{s_{2}}\frac{a_{3}}{s_{3}}\frac{m}{s_{4}%
}=\frac{n}{s}$. Therefore there exists an $s^{\prime }\in S$ such that $%
s^{\prime }sa_{1}a_{2}a_{3}m=s^{\prime }s_{1}s_{2}s_{3}s_{4}n\in N$. So $%
a_{1}a_{2}a_{3}\left( s^{\ast }m\right) \in N$ for $s^{\ast }=s^{\prime }s$.
Since $N$ is a classical 2-absorbing submodule we get $a_{1}a_{2}\left(
s^{\ast }m\right)\in N$ or $a_{1}a_{3}\left( s^{\ast }m\right) \in N$ or $%
a_{2}a_{3}\left( s^{\ast }m\right) \in N$. Thus $\frac{a_{1}a_{2}m}{%
s_{1}s_{2}s_{4}}=\frac{a_{1}a_{2}( s^{\ast }m)}{s_{1}s_{2}s_{4}s^{\ast}}\in
S^{-1}N$ or $\frac{a_{1}a_{3}m}{s_{1}s_{3}s_{4}}\in S^{-1}N$ or $%
\frac{a_{2}a_{3}m}{s_{2}s_{3}s_{4}}\in S^{-1}N$.

(2) Assume that $S^{-1}N$ is a classical 2-absorbing submodule of $S^{-1}M$ and $Z_R(M/N)\cap S=\emptyset$.
Let $a,b,c\in R$ and $m\in M$ such that $abcm\in N$. Then $\frac{a}{1}\frac{b}{1}\frac{c}{1}\frac{m}{1}\in S^{-1}N$.
Therefore $\frac{a}{1}\frac{b}{1}\frac{m}{1}\in S^{-1}N$ or $\frac{a}{1}\frac{c}{1}\frac{m}{1}\in S^{-1}N$
or $\frac{b}{1}\frac{c}{1}\frac{m}{1}\in S^{-1}N$. We may assume that $\frac{a}{1}\frac{b}{1}\frac{m}{1}\in S^{-1}N$.
So there exists $u\in S$ such that $uabm\in N$. But $Z_R(M/N)\cap S=\emptyset$, whence $abm\in N$. 
Consequently $N$ is a classical 2-absorbing submodule of $M$.
\end{proof}

Let $R_i$ be a commutative ring with identity and $M_i$ be an $R_i$-module,
for $i = 1, 2$. Let $R=R_{1}\times R_{2}$. Then $M=M_{1}\times M_{2}$ is an $%
R$-module and each submodule of $M$ is in the form of $N=N_{1}\times N_{2}$ for
some submodules $N_1$ of $M_1$ and $N_2$ of $M_2$.
\begin{theorem}\label{product1}
Let $R=R_{1}\times R_{2}$ be a decomposable ring and $M=M_{1}\times M_{2}$ be an $R$-module where
$M_{1}$ is an $R_{1}$-module and $M_{2}$ is an $R_{2}$-module. Suppose that $N=N_{1}\times N_{2}$ is
a proper submodule of $M$. Then the following conditions are equivalent:

\begin{enumerate}
\item $N$ is a classical 2-absorbing submodule of $M$;

\item Either $N_{1}=M_{1}$ and $N_{2}$ is a classical 2-absorbing submodule of $M_{2}$
or $N_{2}=M_{2}$ and $N_{1}$ is a classical 2-absorbing submodule of $M_{1}$
or $N_{1},$ $N_{2}$ are classical prime submodules of $M_{1}$, $M_{2},$ respectively.
\end{enumerate}
\end{theorem}

\begin{proof}
(1)$\Rightarrow$(2) Suppose that $N$ is a classical 2-absorbing submodule
of $M$ such that $N_2=M_2$. From our hypothesis, $N$ is proper, so $N_{1}\not=M_{1}.$ Set $%
M^{\prime }=\frac{M}{\{0\}\times M_{2}}$. Hence $N^{\prime }=\frac{N}{%
\{0\}\times M_{2}}$ is a classical 2-absorbing submodule of $M^{\prime }$ by
Corollary \ref{quo}. Also observe that $M^{\prime }\cong M_{1}$ and $%
N^{\prime }\cong N_{1}$. Thus $N_{1}$ is a classical 2-absorbing submodule of $%
M_{1}.$ Suppose that $N_{1}\not=M_{1}$ and $N_{2}\not=M_{2}$. We show that
$N_1$ is a classical prime submodule of $M_1$. Since $N_{2}\not=M_{2}$, there exists
$m_2\in M_2\backslash N_2$. Let $abm_1\in N_1$ for some $a,b\in R_1$ and $m_1\in M_1$.
Thus $(a,1)(b,1)(1,0)(m_1,m_2)=(abm_1,0)\in N=N_1\times N_2$. So either $(a,1)(1,0)(m_1,m_2)=(am_1,0)\in N$
or $(b,1)(1,0)(m_1,m_2)=(bm_1,0)\in N$. Hence either $am_1\in N_1$ or $bm_1\in N_1$ which shows that
$N_1$ is a classical prime submodule of $M_{1}$. Similarly we can show that $N_2$ is a classical prime submodule of $M_2$. \newline
(2)$\Rightarrow$(1) Suppose that $N=N_{1}\times M_{2}$ where $N_{1}$ is a classical 2-absorbing 
(resp. classical prime) submodule of $M_{1}$.
Then it is clear that $N$ is a classical 2-absorbing (resp. classical prime) 
submodule of $M$. Now, assume that $N=N_{1}\times N_{2}$ where $N_{1}$ and $N_{2}$ are classical prime
submodules of $M_{1}$ and $M_{2}$, respectively. Hence $(N_{1}\times
M_{2})\cap (M_{1}\times N_{2})=N_{1}\times N_{2}=N$ is a classical 2-absorbing
submodule of $M$, by Proposition \ref{intersection}.
\end{proof}

\begin{lemma}\label{product2} 
Let $R=R_1\times R_2\times\cdots\times R_n$ be a decomposable ring and 
$M=M_1\times M_2\times\cdots\times M_n$ be an $R$-module where for every $1\leq i\leq n$,
$M_{i}$ is an $R_{i}$-module, respectively. A proper submodule $N$ of $M$ is a 
classical prime submodule of $M$ if and only if $N=\times_{i=1}^{n}N_i$ such that for
some $k\in\{1,2,...,n\}$, $N_k$ is a classical prime submodule of $M_k$, and $%
N_i=M_i$ for every $i\in\{1,2,...,n\}\backslash\{k\}$.
\end{lemma}

\begin{proof}
($\Rightarrow$) Let $N$ be a classical prime submodule of $M$.
We know $N=\times_{i=1}^{n}N_i$ where for every $1\leq i\leq n$, $N_i$ is
a submodule of $M_i$, respectively. Assume that $N_r$ is a proper submodule of $M_r$
and $N_s$ is a proper submodule of $M_s$ for some $1\leq r<s\leq n$. So,
there are $m_r\in M_r\backslash N_r$ and $m_s\in M_s\backslash N_s$.
Since 
$$
(0,\dots,0,\overbrace{1_{R_r}}^{r\mbox{-th}},0,\dots,0)(0,\dots,0,\overbrace{%
1_{R_s}}^{s\mbox{-th}},0,\dots,0)(0,\dots,0,\overbrace{m_r}^{r\mbox{-th}},0,\dots,0,
\overbrace{m_s}^{s\mbox{-th}},0,\dots,0)$$$$
=(0,\dots,0)\in N,
$$
then either 
$$(0,\dots,0,\overbrace{1_{R_r}}^{r\mbox{-th}},0,\dots,0)(0,\dots,0,\overbrace{m_r}^{r\mbox{-th}},0,\dots,0,
\overbrace{m_s}^{s\mbox{-th}},0,\dots,0)$$$$=(0,\dots,0,\overbrace{m_r}^{r\mbox{-th}},0,\dots,0)\in N$$
or $$(0,\dots,0,\overbrace{1_{R_s}}^{s\mbox{-th}},0,\dots,0)(0,\dots,0,\overbrace{m_r}^{r\mbox{-th}},0,\dots,0,
\overbrace{m_s}^{s\mbox{-th}},0,\dots,0)$$$$=(0,\dots,0,\overbrace{m_s}^{s\mbox{-th}},0,\dots,0)\in N,$$ which
is a contradiction. Hence exactly one of the $N_i$'s is proper, say $N_k$.
Now, we show that $N_k$ is a classical prime submodule of $M_k$. Let $abm_k\in
N_k $ for some $a,b\in R_k$ and $m_k\in M_k$. Therefore 
$$
(0,\dots,0,\overbrace{a}^{k\mbox{-th}},0,\dots,0)(0,\dots,0,\overbrace{b}^{k%
\mbox{-th}},0,\dots,0)(0,\dots,0,\overbrace{m_k}^{k%
\mbox{-th}},0,\dots,0)$$
$$=(0,\dots,0,\overbrace{abm_k}^{k\mbox{-th}},0,\dots,0)\in
N,$$
and so 
$$(0,\dots,0,\overbrace{a}^{k\mbox{-th}},0,\dots,0)(0,\dots,0,\overbrace{m_k}^{k%
\mbox{-th}},0,\dots,0)=(0,\dots,0,\overbrace{am_k}^{k\mbox{-th}},0,\dots,0)\in
N$$ or $$(0,\dots,0,\overbrace{b}^{k%
\mbox{-th}},0,\dots,0)(0,\dots,0,\overbrace{m_k}^{k%
\mbox{-th}},0,\dots,0)=(0,\dots,0,\overbrace{bm_k}^{k\mbox{-th}},0,\dots,0)\in
N.$$ Thus $am_k\in N_k $ or $bm_k\in N_k $ which implies that $N_k$ is a classical prime submodule of $M_k$.\newline
($\Leftarrow$) Is easy.
\end{proof}

\begin{theorem}\label{product3} 
Let $R=R_1\times R_2\times\cdots\times R_n$ $(2\leq n<\infty)$ be a decomposable ring and 
$M=M_1\times M_2\times\cdots\times M_n$ be an $R$-module where for every $1\leq i\leq n$,
$M_{i}$ is an $R_{i}$-module, respectively. For a proper submodule $N$ of $M$ the
following conditions are equivalent:

\begin{enumerate}
\item $N$ is a classical 2-absorbing submodule of $M$;

\item Either $N=\times^n_{t=1}N_t$ such that for some $k\in\{1,2,...,n\}$, $%
N_k$ is a classical 2-absorbing submodule of $M_k$, and $N_t=M_t$ for
every $t\in\{1,2,...,n\}\backslash\{k\}$ or $N=\times_{t=1}^nN_t$ such that
for some $k,m\in\{1,2,...,n\}$, $N_k$ is a classical prime submodule of $M_k$, 
$N_m$ is a classical prime submodule of $M_m$, and $N_t=M_t$ for every $%
t\in\{1,2,...,n\}\backslash\{k,m\}$.
\end{enumerate}
\end{theorem}

\begin{proof}
We argue induction on $n$. For $n=2$ the result holds by Theorem \ref{product1}.
Then let $3\leq n <\infty$ and suppose that the result is valid when $%
K=M_1\times\cdots\times M_{n-1}$. We show that the result holds when $%
M=K\times M_n$. By Theorem \ref{product1}, $N$ is a classical 2-absorbing
submodule of $M$ if and only if either $N=L\times M_n$ for some
classical 2-absorbing submodule $L$ of $K$ or $N=K\times L_n$ for some
classical 2-absorbing submodule $L_n$ of $M_n$ or $N=L\times L_n$ for
some classical prime submodule $L$ of $K$ and some classical prime submodule $%
L_n $ of $M_n$. Notice that by Lemma \ref{product2}, a proper submodule $L$ of $K$
is a classical prime submodule of $K$ if and only if $L=\times_{t=1}^{n-1}N_t$
such that for some $k\in\{1,2,...,n-1\}$, $N_k$ is a classical prime submodule
of $M_k$, and $N_t=M_t$ for every $t\in\{1,2,...,n-1\}\backslash\{k\}$.
Consequently we reach the claim.
\end{proof}

\vspace{3mm} \noindent \footnotesize 
\begin{minipage}[b]{10cm}
Hojjat Mostafanasab \\
Department of Mathematics and Applications, \\
University of Mohaghegh Ardabili, \\
P. O. Box 179, Ardabil, Iran. \\
Email: h.mostafanasab@uma.ac.ir, \hspace{1mm} h.mostafanasab@gmail.com
\end{minipage}

\vspace{3mm} \noindent \footnotesize
\begin{minipage}[b]{10cm}
\"{U}nsal Tekir\\
Department of Mathematics, \\ 
Marmara University, \\ 
Ziverbey, Goztepe, Istanbul 34722, Turkey. \\
Email: utekir@marmara.edu.tr
\end{minipage}

\vspace{3mm} \noindent \footnotesize
\begin{minipage}[b]{10cm}
K\"{u}r\c{s}at Hakan Oral\\
Department of Mathematics,\\ 
Yildiz Technical University,\\ 
Davutpasa Campus, Esenler, Istanbul, Turkey.\\
Email: khoral@yildiz.edu.tr
\end{minipage}

\end{document}